\documentclass[11pt,a4paper,twoside]{amsart}
\usepackage{amsmath,amssymb,amsfonts,amsthm}
\usepackage{hyperref}

\def\K{{\mathcal{K}}}
\def\m{\stackrel{1}{M}}
\def\mm{\stackrel{2}{M}}

\def\K{\stackrel{1}{K}}
\def\KK{\stackrel{2}{K}}
\def\KKK{\stackrel{k}{K}}
\def\st{\stackrel}

\def\to{\longrightarrow}
\def\mto{\longmapsto}
\def\a{\alpha}
\def\b{\beta}
\def\e{\eta}
\def\g{\gamma}
\def\p{\phi}
\def\P{\Phi}
\def\G{\Gamma}
\def\s{\psi}
\def\S{\Psi}

\def\eps{\epsilon}
\def\r{\rho}
\def\o{\circ}

\def \dar{\times}
\def \lim{\varprojlim}

\def\N{\mathbb{N}}

\def\F{\mathbb{F}}
\def\R{\mathbb{R}}
\def\E{\mathbb{E}}

\newtheorem{The}{Theorem}[section]
\newtheorem{Pro}[The]{Proposition}
\newtheorem{Lem}[The]{Lemma}
\newtheorem{Cor}[The]{Corollary}
\theoremstyle{definition}
\newtheorem{Def}[The]{Definition}
\newtheorem{Rem}[The]{Remark}
\newtheorem{Examp}[The]{Example}
\thanks{2010 Mathematical Subject Classification.
   Primary 58B20; Secondary 58A05.}

\begin{document}
\title{Higher order tangent bundles}
\author{Ali Suri }
\address{Department  of Mathematics, Faculty of sciences, \\
Bu Ali Sina University, Hamedan 65178, Iran.}
\email{ali.suri@yahoo.com \& a.suri@math.iut.ac.ir \& ali.suri@gmail.com}
\maketitle {\hspace{2.5cm}}

\begin{abstract}
The tangent bundle $T^kM$ of order $k$,  of a smooth Banach
manifold $M$ consists of all equivalent classes of curves that
agree up to their accelerations of order $k$. For a Banach
manifold $M$ and a natural number $k$ first we determine a smooth
manifold structure on $T^kM$ which also offers a fiber bundle
structure for $(\pi_k,T^kM,M)$. Then we introduce a particular
lift of linear connections on $M$  to geometrize $T^kM$ as a
vector bundle over $M$. More precisely based on this lifted
nonlinear connection we prove that $T^kM$ admits a vector bundle
structure over $M$ if and only if $M$ is endowed with a linear
connection. As a consequence applying this vector bundle structure
we lift  Riemannian metrics and Lagrangians from $M$ to $T^kM$. Also, using the
projective limit techniques, we declare a generalized Fr\'{e}chet
vector bundle structure for $T^\infty M$ over $M$.\\

\textbf{Keywords}: Banach manifold;
Linear connection; Connection map; Higher order tangent bundle;
Fr\'{e}chet manifold; lifting of Riemannian metrics,  Lagrangians.
\end{abstract}

\pagestyle{headings} \markright{Higher order tangent bundles.}

\section{Introduction}

Higher order tangent bundles $T^kM$ of  a smooth manifold $M$ as
the space of all equivalent classes of curves that agree up to
their accelerations of order $k$, is a natural generalization of
the notion of tangent bundle $TM$. Higher order geometry had
witnessed a wide interest due to the works of  Bucataru, Dodson,
Miron, Morimoto and others \cite{Ioan}, \cite{Dod-Gal},
\cite{Miron}, \cite{Morimoto}. The geometry of $T^kM$ in the
finite dimensional case is developed by Miron and his school
\cite{Miron}. They studied higher order Lagrangians and also
prolongation of Riemannian metrics, Finsler structures and
 Lagrangians to $T^kM$.

However,  even for the case of $n=2$, constructing a vector bundle
(for abbreviation v.b.) structure on $T^2M$ over $M$ is not as
evident as in the case of $TM$. More precisely sometimes it is
impossible to define a v.b. structure on $T^2M$.  Dodson and
Galanis \cite{Dod-Gal} proved that for a Banach manifold $M$,
$T^2M$ can be thought of as a Banach vector bundle over $M$ if and
only if $M$ is endowed with a linear connection. The author proved
the same result in a different way to geometrize the bundle of
accelerations with more tools like second order covariant
derivative, exponential mapping and an appropriate second order
Lie bracket \cite{ali}.
In this paper, in order to geometrize  higher order tangent
bundles, first we introduce an special lifted connection which
will plays a pivotal role in our main theorem. Then  we prove that
for any $k\in \N$, $T^kM$ can be thought of as a v.b. over $M$ with
the structure group $ GL(\E^k)$ if and only if $M$
admits a linear connection. Furthermore  this result considerably
eases constructing a v.b. structure on $\pi_{ji}:T^jM\to T^iM$ for
$j>i$. We shall also show that if for some $k\geq 2$,
$(\pi_k,T^kM,M)$ becomes a v.b.  isomorphic to $\oplus_{i=1}^kTM$, then for any $n\in\N\cup \{\infty\}$,
$T^nM$ admits a v.b. structure over $M$. More precisely in the
case of infinite order, $T^\infty M$ becomes a Fr\'{e}chet
manifold which may be thought of as a generalized v.b. over $M$.
Moreover the structure group becomes a generalized Fr\'{e}chet lie
group which represents the advantage of using projective limit
techniques.

Another old problem in geometry is that of prolongation of
Riemannian and Lagrangian structures to the tangent bundles $T^kM$. These problems
can also be solved as  a consequence of our main theorem. Finally using the restricted symplectic group we propose an example to support our theory. However, for more examples we refer to \cite{iso Osck} and \cite{hfb}.

Trough this paper all the maps and manifolds are assumed to be smooth,
but, except  in section \ref{Section T infty M}, a lesser degree of differentiability can be assumed. Whenever partition of unity is necessary, we assume that our manifolds are partitionable \cite{Lang, ali}.

Most of the results of this paper are novel even for the case that
$M$ is a finite dimensional manifold.
\vspace{-12mm}
\section{Preliminaries}
Let M be a manifold, possibly infinite dimensional, modelled on
the Banach space $\E$. For any $x_0\in M$ define \vspace{-2mm}
\begin{equation*}
C_{x_0}:=\{\g:(-\eps,\eps)\to M~;~ \g(0)=x_0  ~ \textrm{and}~ \g
~\textrm{is smooth }\}.
\end{equation*}
As a natural extension of the  tangent bundle $TM$ define the
following equivalence relation. The curves  $\g_1,\g_2 \in C_{x_0}$
are said to be $k$-equivalent, denoted by $\g_1\approx_{x_0}^k\g_2$,
if and only if $\g_1^{(j)}(0)=\g_2^{(j)}(0)$  for $1\leq j\leq k$.
Define ${T}^k_{x_0}M:=C_{x_0}/\approx_{x_0}^k$ and the \textbf{
tangent bundle of order  $k$ or $k$-osculating bundle} of M to be
${T}^kM:=\bigcup_{x\in M}{T}^k_{x}M$. Denote by $[\g,{x_0}]_k$ the
representative of the equivalence class containing $\g$  and define
the canonical projection ${\pi}_k:T^kM\to M$ which projects
$[\g,{x_0}]_k$ onto $ x_0$.

Let $\mathcal{A}=\{(U_\a,\s_\a)\}_{\a\in I}$ be a $C^\infty$ atlas
for $M$. For any $\a\in I$ define
\begin{eqnarray*}
\S_\a^k:{\pi_k}^{-1}(U_\a)&\to& \s_\a(U_\a)\times\E^k\\
{[\gamma,{x_0}]_k}&\mto& \big( (\s_\a
\circ\g)(0),(\s_\a\circ\g)'(0),...,\frac{1}{k!}(\s_\a\circ\g)^{(k)}(0)\big)
\end{eqnarray*}
%
%
%
%
\begin{The}\label{fibre bundle structure for TkM}
 The family
$\mathcal{B}=\{({\pi_k}^{-1}(U_\a),\S_\a^k)\}_{\a\in I}$ declares
a smooth manifold structure on $T^kM$ which models it on
$\E^{k+1}$.
\end{The}
\begin{proof}
Clearly $\S_\a^k$ is well defined and $\bigcup_{\a\in
I}{\pi_k}^{-1}(U_\a)=T^kM$. $\S_\a^k$ is surjective. In fact  for any
$(x,\xi_1,...,\xi_k)\in\s_a(U_\a)\times \E^k$ the class
$[\g,\s_\a^{-1}(x)]_k$, with $\g:=\psi_\a^{-1}\o\bar\g$ and
$\bar\g(t)=x+t\xi_1+...+t^k\xi_k$, is mapped  to
$(x,\xi_1,...,\xi_k)$ via $\S_\a^k$. It is easy to show that
$\S_\a^k$ is also injective.\\ For any $\a,\b\in I$ with
$U_{\b\a}:=U_\b\cap U_\a\neq{\emptyset}$ the overlap map
\begin{eqnarray*}
\S_{\b\a}^k:=\S_\b^k\o{\S_\a^k}^{-1}:\s_\a(U_{\b\a})\times\E^k \to
\s_\b(U_{\b\a})\times\E^k
\end{eqnarray*}
is given by
\begin{eqnarray*}
&&\S_{\b\a}^k(x,\xi_1,...,\xi_k)=\S_\b^k([\gamma,{x_0}]_k)\\
&=&\big((\s_\b\o\g)(0),(\s_\b\o\g)'(0),...,\frac{1}{k!}(\s_\b\o\g)^{(k)}(0)\big)\\
&=&\big((\s_\b\o\s_\a^{-1}\o\bar\g)(0),(\s_\b\o\s_\a^{-1}\o\bar\g)'(0),\dots,\frac{1}{k!}(\s_\b\o\s_\a^{-1}\o\bar\g)^{(k)}(0)\big)\\
&=&\big(\s_{\b\a}(x),d\s_{\b\a}(x)\xi_1,\dots,\frac{1}{k!}\{d\s_{\b\a}(x)[\bar{\gamma}^{(k)}(0)]\\
&&+\sum_{j_1+j_2=k}a^k_{(j_1,j_2)}d^2\s_{\b\a}(x)[\bar{\gamma}^{(j_1)}(0),\bar{\gamma}^{(j_2)}(0)]\\
&&
+...+d^k\s_{\b\a}(x)(\bar{\gamma}'(0),...,\bar{\gamma}'(0)]\}
\big)\\
&=&\big(\s_{\b\a}(x),d\s_{\b\a}(x)\xi_1,d\s_{\b\a}(x)(\xi_2)+\frac{1}{2}d^2\s_{\b\a}(x)(\xi_1,\xi_1),\dots\\
&&,\frac{1}{k!}\{d\s_{\b\a}(x)[k!\xi_k]+\sum_{j_1+j_2=k}a^k_{(j_1,j_2)}d^2\s_{\b\a}(x)[
j_1!\xi_{j_1},j_2!\xi_{j_2}]\\
&&+...+d^k\s_{\b\a}(x)(\xi_1,\dots,\xi_1)\}
\big)\\
\end{eqnarray*}
where $\s_{\b\a}=\s_\b\o\s_\a^{-1}$ and
$\bar\gamma(t)=x+t\xi_1+...+t^k\xi_k$ as before. Moreover we  used
the following explicit formula  for  the chain
rule of order $k$
%
%
\begin{eqnarray}\label{chain rule of order k lloyd}
&(\s_{\b\a}\o \bar{\gamma})^{(k)}(0)=d^k(\s_{\b\a}\o
\bar{\gamma})(0)(1,\dots,1)&\\
\nonumber &=\sum_{i=1}^k\sum \frac{k!}{j_1!\dots j_i! m_1!\dots m_k!}
d^i\s_{\b\a}(\bar{\gamma}(0))[\bar{\gamma}^{(j_1)}(0),\dots,\bar{\gamma}^{(j_i)}(0)]&
%
%
%
%
%
\end{eqnarray}
where the second sum is over all (ordered) $i$-tuples $(j_1,\dots,j_i)$
of positive integers such that $j_1+...+j_i=k$ and $m_1$ of the numbers $l_1,\dots , l_i$ are equal to $1$, $m_2$ are equal to $2$ and so on (\cite{Averbuh}, p. 234, \cite{Lloyd}, p. 359). The coefficient  $ \frac{k!}{j_1!\dots j_i! m_1!\dots m_k!}$ will henceforth be denoted by
$a^k_{(j_1,\dots,j_i)} $.
\end{proof}
%
%
%
%
Due to the transition functions of the bundle $({\pi}_k,T^kM,M)$, we
can see that generally it is a smooth fibre bundle.

To compute the local forms for the change of charts of $TT^kM$ on
overlaps  we remind some facts about fibre bundles. Let
$p:E\longrightarrow M$ be a smooth Banach fibre bundle with fibres
diffeomorphic to the Banach manifold $F$ and the Banach spaces
$\mathbb{E}$ and $\mathbb{B}$ are  the model spaces for $F$ and
$M$ respectively. Suppose that
$\Phi=(\phi,\bar\phi):E|_U\longrightarrow \phi(U)\times \bar\phi(E|_U)\subseteq\phi(U)\times
\mathbb{E}$ be a local trivialization where $(U,\phi)$ is a chart
of $M$ and let $(\Psi=(\psi,\bar\psi),V)$ be another local
trivialization with $U\cap V\neq{\emptyset}$. Then $\Psi\circ
\Phi^{-1}(x,\xi)=((\psi\circ\phi^{-1})(x),G_{\Psi\Phi}(x,\xi))$
where $G_{\Psi\Phi}:U\cap V\times \bar\p(E|_{U\cap V})\subseteq U\cap V\times\mathbb{E}\longrightarrow
\mathbb{E}$ is smooth. The canonical induced trivialization for
$TE$ is
\begin{eqnarray}
T(\Psi\circ
\Phi^{-1})(x,\xi;y,\eta)&=&\big((\psi\circ\phi^{-1})(x),G_{\Psi\Phi}(x,\xi),d(\psi\circ\phi^{-1})(x)y \\
\nonumber&&,\partial_1G_{\Psi\Phi}(x,\xi)y+\partial_2G_{\Psi\Phi}(x,\xi)\eta\big).
\end{eqnarray}
for any $(x,\xi,y,\eta)\in
\phi(U\cap V)\times\bar\p(E|_{U\cap V})\times\mathbb{B}\times\mathbb{E}$.
(Throughout this paper the symbol $\partial_i$ denotes the partial derivative
with respect to the $i$-th variable.)\\
Now
using the transition functions for the bundle $\pi_k:T^kM\to M$ we can
compute the transformation rule of natural charts of $T{T}^kM$.
For any $u=(x,\xi_1,...,\xi_k)\in U_\a\times \E^k$ and
$(y,\e_1,...\e_k)\in\E^{k+1}$ we have
\begin{equation}\label{TE}
T\S_{\b\a}^k(u;y,\e_1,\dots,\e_k)=
\textbf{\Big(}\S_{\b\a}^k(u);d\s_{\b\a}(x)y~,\bar{\e}_1,\dots,\bar{\e}_k)\textbf{\Big)}
\end{equation}
where
\begin{eqnarray*}
&\bar{\e}_i=\frac{1}{i!}\{d\s_{\b\a}(x)(i!\e_i)+\sum_{j_1+j_2=i}a^i_{(j_1,j_2)}
[d^2\s_{\b\a}(x)(j_1!\e_{j_1},j_2!\xi_{j_2})&\\
&+d^2\s_{\b\a}(x)(j_1!\xi_{j_1},j_2!\e_{j_2})]
+\dots+ id^i\s_{\b\a}(x)(\xi_1,\xi_1,\dots,\xi_1,\e_1) &\\
&+ d^2\s_{\b\a}(x)(i!\xi_i,y)+\sum_{j_1+j_2=i}a^i_{(j_1,j_2)}  d^3\s_{\b\a}(x)(j_1!\xi_{j_1},j_2!\xi_{j_2},y)&\\
&+\dots+d^{i+1}\s_{\b\a}(x)(\xi_1,\xi_1,\dots,\xi_1,y)\}&\\
&=\frac{1}{i!}\frac{\partial^{i+1}}{\partial s\partial t^i}(\s_{\b\a}\o \bar{c})(t,s)|_{t=s=0}&
\end{eqnarray*}
and
\begin{eqnarray*}
\bar{c}:(-\eps,\eps)^2    &    \to       &   \s_\a(U_\a)\subseteq \E\\
(t,s)&\mto& x+sy+\sum_{j=1}^it^j(\xi_j+\eta_j).
\end{eqnarray*}
%
%
\section{Tangent bundle of order $k$ for Banach manifolds}
This section includes two parts. In the first part  for any linear
connection $\nabla$ on $M$, in the sense of Vilms \cite{Vilms}, we
determine a v.b. morphism $K:T{T}^kM\to\oplus_{i=1}^kTM$ which may
be thought as an special lift of connections. This kind of lift,
named connection maps by Bucataru \cite{Ioan}, induces nonlinear
connections on $T^kM$. Then using $K$ as a key, we determine a
v.b. structure on $\pi_k: T^kM\to M$ which is followed with a
suitable converse.
\subsection{Connection maps in higher order geometry}

Consider the
$C^\infty(T^kM)$-linear map $J:\mathfrak{X}(T^kM)\to \mathfrak{X}(T^kM)$  such that
locally on a chart $({\pi}_k^{-1}(U_\a),\S_\a^k)$ is given by\\
$$J_{\a}(u;y,\e_1,\dots,\e_k)=(u;0,y,\e_1,\dots,\e_{k-1}).$$
for any $u=(x,\xi_1,\dots,\xi_k)\in T^kM$ and every
$(u;y,\e_1,\dots,\e_k)\in T_uT^kM$.
\begin{Def}\label{Def connection map}
A connection map on ${T}^kM$
is a vector bundle morphism\vspace{-2mm}
$$K=(\K ,\KK...,\KKK):T{T}^kM\to\Big(\oplus_{i=1}^kTM,\oplus_{i=1}^k\tau_M,\oplus_{i=1}^kM\Big)\vspace{-2mm}$$
such that for any $1\leq a\leq k-1$, $\KKK\o
J^a=\stackrel{k-a}{K}$ and $\KKK\o J^k={\pi_k}_*$
\end{Def}
Bucataru defined this connection map in the finite dimensional context
\cite{Ioan}. Note that \vspace{-1mm}

\begin{equation*}
\stackrel{a}{K}=\KKK\o J^{k-a}=\KKK\o J^{k-a-1}\o
J=\stackrel{a+1}{K}\o J
\end{equation*}\vspace{-1mm}
and $\stackrel{ a}{K}\o J^a=(\KKK\o J^{k- a})\o J^a={\pi_k}_*$.

\begin{Lem} Locally on a chart $(\pi_k^{-1}(U_\a),\S_\a^k)$ the connection map
$$\oplus_{i=1}^k\S_\a^1 \o K \o T{\S_\a^k}^{-1}:=K_\a=(\K_\a,...,\KKK_\a )$$
at $(u;y,\e_1,...,\e_k)\in
T_uT^kM$ is given by
\begin{eqnarray}\label{local form of K}
&K|_{\a}(u;y,\e_1,...,\e_k)&\\
\nonumber&=\bigoplus_{i=1}^k\Big(x,\e_i+\m_\a(u)\e_{i-1}+
\mm_\a(u)\e_{i-2}+...+\stackrel{i}{M}_\a(u) y\Big)&
\end{eqnarray}\vspace{-2mm}
\end{Lem}
%
%
\begin{proof}
Since $K$ is bundle morphism there are local maps
$\stackrel{i}{M}_\a:U_\a\dar\E^k\to L(\E,\E)$, $1\leq i\leq k$, such that \vspace{-2mm}
\begin{eqnarray*}
K_\a(u;y,0,...,0)=(x,\m_\a(u)y)\oplus(x,\mm_\a(u)y)\oplus...\oplus
(x,\stackrel{k}{M}_\a(u)y).
\end{eqnarray*}\vspace{-2mm}
Moreover due to the facts $\stackrel{a}{K}\o J^a={\pi_k}_*$ and
$\stackrel{a}{K}=\stackrel{a+1}{K}\o J$ we  get
\begin{eqnarray*}
K_\a(u;0,\e_1,0,...,0)&=&K_\a\o
J(u;\e_1,0,...,0)\\
&=&(x,\e_1)\oplus(x,\m(u)\e_1)\oplus...\oplus(x,\stackrel{k-1}{M}(u)\e_1)
\end{eqnarray*}\vspace{-2mm}
and likewise
\begin{eqnarray*}
K_\a(u;0,0,\e_2,0,...,0)=(x,0)\oplus(x,\e_2)\oplus(x,\m(u)\e_2)...\oplus(x,\stackrel{k-2}{M}(u)\e_2).
\end{eqnarray*}
which  completes the proof.
\end{proof}
%
%
For any $\a,\b\in I$ with $U_{\b\a}\neq\emptyset$ the
compatibility condition for $\stackrel{i}{M}_\a$ and
$\st{i}{M}_\b$, $1\leq i\leq k$, on the overlaps comes from the
fact  $\oplus_{i=1}^k T\s_{\b\a}\o K_\a = K_\b\o T\S_{\b\a}^k$. We apply equality
(\ref{TE}) and the local form of $K$ to obtain \vspace{-1.5mm}
\begin{eqnarray}\label{campatibility for Ma and Mb}
&&d\s_{\b\a}(x)[\e_i+\m_\a(u)\e_{i-1}+\dots+\st{i}{M}_\a(u)y]\\
\nonumber&&=\bar{\e}_i+\m_\b(\bar{u})\bar{\e}_{i-1}+\mm_\b(\bar{u})\bar{\e}_{i-2}+\dots+\st{i}{M}_\b(\bar{u})\bar{y}
\end{eqnarray}
for any $u\in {T}^kM$ and any $(u,y,\e_1,\dots,\e_k)\in T_uT^kM$.
Moreover $\bar{y}=d\s_{\b\a}(x)y$ and $\bar{\eta_1},\dots, \bar{\eta}_k$ are as in the equation (\ref{TE}).
%
%
\begin{The}\label{lifted connection to osckM}
Let $\nabla$ be  a (linear) connection on $M$ with the local
components $\{\G_\a\}_{\a\in I}$. There exists an induced connection
map on ${T}^kM$ with the following local components.
\begin{eqnarray*}
&&\m_\a(x,\xi_1)y=\Gamma_\a(x,\xi_1)y\\
&&\mm_\a(x,\xi_1,\xi_2)y=\frac{1}{2}\Big(\sum_{i=1}^2\partial_i\m_\a(x,\xi_1)(y,i\xi_i)+\m_\a(x,\xi_1)[\m_\a(x,\xi_1)y]\Big),\\
&&\vdots\\
&&\st{k}{M}_\a(x,\xi_1,...,\xi_k)y=\frac{1}{k}\Big(\sum_{i=1}^{k}\partial_i\st{k-1}{M}_\a(x,\xi_1,...,\xi_{k-1})(y,i\xi_i)\\
&&\hspace{3.7cm}+\m_\a(x,\xi_1)[\st{k-1}{M}_\a(x,\xi_1,...,\xi_{k-1})y]\Big)
\end{eqnarray*}
The proof of the compatibility condition  for $\st{i}{M}_\a$ and
$\st{i}{M}_\b$ on overlaps can be found in \cite{iso Osck}.
\end{The}
Note that  kernel of $K$ is a distribution, say $H\pi_k$, on
$T^kM$ complementary to the canonical vertical distribution
$V\pi_k$. The horizontal distribution $H\pi_k$ is called the
nonlinear connection associated to $K$ \cite{Miron}.
%
%
%
%
%
%
\subsection{${T}^kM$ as a vector bundle}
For $k\geq 2$, the bundle structure defined in theorem \ref{fibre
bundle structure for TkM} is quite far from being a v.b. due to
the complicated nonlinear transition functions. Here we propose a
v.b. structure on $\pi_k:T^kM\to M$ which makes it a smooth v.b.
isomorphic  to $k$ copies of $TM$. The converse  of the problem is
also true i.e. a v.b. structure on $T^kM$ isomorphic to
$\oplus_{i=1}^kTM$, for $k\geq 2$, yields a linear connection on
$M$. Moreover it will be shown that if for some integer $k\geq2$,
$T^kM$ becomes a v.b over $M$ with the before-mentioned property
then   $T^iM$ also admits a v.b. structure over $M$ for any $i\in
\N$. These v.b. structures simplifies the study of higher tangent
bundles. For example as a consequence, we propose a lifted
Riemannian metric to $T^kM$ which only depends to the original
given metric on the base manifold $M$.
%
%
%
%
%
%
%
\begin{The}\label{osc k admits a vb}
Let $\nabla$ be a (linear) connection on $M$ and $K$ be the induced
connection map introduced in theorem \ref{lifted connection to
osckM}. The following trivializations  define a vector bundle
structure on $\pi_k:{T}^kM\to M$ with the structure group
$GL(\E^k)$.
\begin{eqnarray*}
\P_\a^k:\pi_k^{-1}(U_\a)&\to& \s_\a(U_\a)\times \E^k\\
{[\g,x]}_k&\mto&(\g_\a(0),\g_\a'(0),z^2_\a([\g,x]_k),\dots,z^k_\a([\g,x]_k))
\end{eqnarray*}
where $\g_\a=\s_\a\o \g$ and
\begin{eqnarray*}
z^2_\a([\g,x]_k)&=&\frac{1}{2}\Big\{\frac{1}{1!}\g_\a''(0)+\m_\a[\g_\a(0),\g_\a'(0)]\g_\a'(0)\Big\},\dots,\\
%
%
%
z^k_\a([\g,x]_k)&=&\frac{1}{k}\Big\{\frac{1}{(k-1)!}\g_\a^{(k)}(0)+\frac{1}{(k-2)!}\m_\a[\g_\a(0),\g_\a'(0)]\g_\a^{(k-1)}(0)+\dots\\
&&+\stackrel{k-1}{M}_\a
[\g_\a(0),\g_\a'(0),\dots,\frac{1}{(k-1)!}\g_\a^{(k-1)}(0)]\g_\a'(0)\Big\}.
\end{eqnarray*}
Moreover for any $(x,\xi_1,\dots\xi_k)\in \s_\a(U_{\a\b})\times\E^k$ we have
\begin{equation*}
\P_\b^k\o{{\P}_\a^k}^{-1}(x,\xi_1,\xi_2,\dots,\xi_k)=\Big(\s_{\b\a}(x),d\s_{\b\a}(x)\xi_1,...,d\s_{\b\a}(x)\xi_k\Big).
\end{equation*}
\end{The}
\begin{proof}

Clearly for any $\a\in I$, $\P_\a^k$ is well defined and
injective.

For any $(x,\xi_1,...,\xi_k)\in \s_\a(U_\a)\times
\E^{k}$ we  show that there exits a curve $\g$ in $M$ such that
$\P_\a^k([\g,x]_k)=(x,\xi_1,...,\xi_k)$. If
$\bar\g_2(t)=x+t\xi_1+\frac{t^2}{2}\{2\xi_2-\m_\a(x,\xi_1)\xi_1\}$
then $z^2_\a([\g_2,x]_k)=\xi_2$ where
$\g_2(t)=\s_\a^{-1}\o{\bar\g}_2(t)$. Now by induction we assume
that for $i-1<k$ there exists $\bar\g_{i-1}$, a polynomial of
degree $i-1$, such that $\g_{i-1}=\s_\a^{-1}\o\bar\g_{i-1}$ and
$z^{j}_\a([\g_j,x]_k)=\xi_j$ for $2\leq j\leq i-1$. Now $\bar\g_i$
is defined by setting
\begin{eqnarray*}
&{\bar\g}_i(t)=\bar\g_{i-1}(t)+\frac{t^i}{i}\{i\xi_i-\frac{1}{(i-2)!}\m_\a(x,\xi_1)\bar\g^{(i-1)}_{i-1}(0)
-\dots&\\
&\stackrel{i-1}{M}_\a(x,\xi_1,\frac{1}{2!}\bar\g_2^{(2)}(0),...,\frac{1}{(i-1)!}\bar\g_{i-1}^{(i-1)}(0))\xi_1\}&
\end{eqnarray*}
and $\g_i(t)=\s_\a^{-1}\o\bar\g_i(t)$. Set $\g=\g_k$. As a result
$z^k([\g,x]_k)=\xi_k$  which means that
$\P_\a^k([\g,x]_k)=(x,\xi_1,...,\xi_k)$. Since
$proj_1\o\P_\a^k=\pi_k$ it follows that  $T^kM$ is a fibre bundle.\\
For any $\a,\b\in I$ with $U_{\b\a}\neq{\emptyset}$ we prove that
$\P_{\b\a}^k:=\P_\b^k\o{\P_\a^k}^{-1}$ induces a linear
isomorphism between fibers. In fact we have
\begin{eqnarray*}
&&\P_{\b\a}^k(x,\xi_1,\xi_2,\dots,\xi_k)=\P_\b^k([\g,x]_k)\\
&&=\Big( (\s_{\b}\o\g)(0),      (\s_{\b}\o\g)'(0),z_\b^2([\g,x]_k),\dots,z_\b^k([\g,x]_k)    \Big) .
\end{eqnarray*}
\textbf{Step 1.} Since $\s_{\b}\o\g=\s_{\b}\o\s_\a^{-1}\o\s_\a\o\g=\s_{\b\a}\o\bar{\g}$, for any $2\leq i\leq k$ we get
\begin{eqnarray*}
i z_\b^i([\g,x]_k)&= & \frac{(\s_{\b\a}\o\bar{\g})^{(i)}(0)}{(i-1)!}   + \m_\b  (\s_{\b\a}(x),(\s_{\b\a}\o\bar{\g})'(0)    )
\frac{(\s_{\b\a}\o\bar{\g})^{(i-1)}(0)}{(i-2)!}    \\
&&+\dots +\stackrel{i-1}{M}_\b     \big(   \s_{\b\a}(x),\dots ,\frac{  (\s_{\b\a}\o\bar{\g})^{(i-1)}(0)   }{(i-1)!}   \big)
(\s_{\b\a}\o\bar{\g})'(0).
\end{eqnarray*}
Using the chain rule formula (\ref{chain rule of order k lloyd}) we conclude that
\begin{eqnarray*}
&&i z_\b^i([\g,x]_k)=  \frac{(\s_{\b\a}\o\bar{\g})^{(i)}(0)}{(i-1)!}   + \m_\b  \big(   \s_{\b\a}(x), (\s_{\b\a}\o\bar{\g})^{\prime}(0)  \big) \frac{(\s_{\b\a}\o\bar{\g})^{(i-1)}(0)}{(i-2)!}    \\
&&+\dots +\stackrel{i-1}{M}_\b     \big(   \s_{\b\a}(x),\dots ,\frac{  (\s_{\b\a}\o\bar{\g})^{(i-1)}(0)  }{(i-1)!}   \big) (\s_{\b\a}\o\bar{\g})'(0)\\
&=&   \frac{1}{(i-1)!}   \{   d\s_{\b\a}(x)(\bar\g^{(i)}(0))        +     \sum_{j_1+j_2=i} a_{(j_1,j_2)}^i d^2\s_{\b\a}(x)[\bar\g^{(j_1)}(0) , \bar\g^{(j_2)}(0)  ]  \\
&&+\dots  + d^{i}\s_{\b\a}(x)[  \bar{\g}'(0),\dots,\bar{\g}'(0)]                            \}\\
&&+ \m_\b  \big(   \s_{\b\a}(x), (\s_{\b\a}\o\bar{\g})^{\prime}(0)  \big) \frac{(\s_{\b\a}\o\bar{\g})^{(i-1)}(0)}{(i-2)!}    \\
&&+\dots +\stackrel{i-1}{M}_\b     \big(   \s_{\b\a}(x),\dots ,\frac{  (\s_{\b\a}\o\bar{\g})^{(i-1)}(0)  }{(i-1)!}   \big) (\s_{\b\a}\o\bar{\g})'(0)
\end{eqnarray*}
\begin{eqnarray*}
&=&      d\s_{\b\a}(x)   \Big[   i\xi_i  - \m_\a(x,\xi_1)\frac{\bar{\g}^{(i-1)}}{(i-2)!}      -
\dots    -        \stackrel{i-1}{M}_\a  \big(   x,\xi_1,\dots,\frac{\bar{\g}^{(i-1)}}{(i-1)!}   \big)    \bar{\g}'(0)  \Big]     \\
&&\frac{1}{(i-1)!}     \{     \sum_{j_1+j_2=i} a_{(j_1,j_2)}^i d^2\s_{\b\a}(x)[\bar\g^{(j_1)}(0) , \bar\g^{(j_2)}(0)  ]         +\dots  \\
&&+ d^{i}\s_{\b\a}(x)[  \bar{\g}'(0),\dots,\bar{\g}'(0)]                            \}\\
&&+ \m_\b  \big(   \s_{\b\a}(x), (\s_{\b\a}\o\bar{\g})^{\prime}(0)  \big) \frac{(\s_{\b\a}\o\bar{\g})^{(i-1)}(0)}{(i-2)!}    \\
&&+\dots +\stackrel{i-1}{M}_\b     \big(   \s_{\b\a}(x),\dots ,\frac{  (\s_{\b\a}\o\bar{\g})^{(i-1)}(0)  }{(i-1)!}   \big) (\s_{\b\a}\o\bar{\g})'(0)\\
\end{eqnarray*}

\textbf{Step 2.}  Setting $x=\bar{\g}(0)$, $\xi_1=\bar{\g}'(0)$, $\dots$, $\xi_{i-1}=\frac{\bar\g^{(i-1)}}{(i-1)!}$, $y=\bar{\g}'(0)$, $\eta_1=\frac{ \bar{\g}^{(2)}(0)}{1!~}$, $\dots$, $\e_{i-1}=\frac{\bar\g^{(i)}}{(i-1)!}$ and
$$\bar{c}_i(t,s)= \bar\g(0)+s\bar\g'(0)+\sum_{l=1}^{i-1 }\frac{t^l}{l!}(\bar\g^{(l)}(0)+s \bar\g^{(l+1)}(0)  ),$$
then, equation (\ref{campatibility for Ma and Mb}) implies that
\begin{eqnarray*}
&&  d\s_{\b\a}(x)   \Big[  \m_\a(x,\xi_1)\frac{\bar{\g}^{(i-1)}}{(i-2)!}      +
\dots    +        \stackrel{i-1}{M}_\a   \big(   x,\xi_1,\dots,\frac{\bar{\g}^{(i-1)}}{(i-1)!}   \big)   \bar{\g}'(0)  \Big]     \\
&=&-d\s_{\b\a}(x)\frac{\bar\g^{(i)}(0)}{(i-1)!} +\frac{\partial^{i}}{\partial s\partial t^{i-1}}\frac{(\s_{\b\a}\o\bar{c}_i)}{(i-1)!}(0,0)\\
&&+           \m_\b \Big(  (\s_{\b\a}\o\bar{c}_i)(0,0)  ,  \frac{\partial}{\partial t}(\s_{\b\a}\o\bar{c}_i)(0,0)  \Big) \frac{\partial^{i-1}}{\partial s\partial t^{i-2}}\frac{(\s_{\b\a}\o\bar{c}_i)}{(i-2)!}(0,0) + \dots\\
&&  + \stackrel{i-1}{M}_\b \Big(  (\s_{\b\a}\o\bar{c}_i)(0,0)  ,\dots ,  \frac{\partial^{i-1}}{\partial t^{i-1}}   \frac{(\s_{\b\a}\o\bar{c}_i)}{(i-1)!}   (0,0)  \Big)  \frac{\partial}{\partial s} (\s_{\b\a}\o\bar{c}_i)   (0,0).
\end{eqnarray*}
It is not hard to check that, for any $1\leq l\leq i-1$, $\frac{\partial^l} {   \partial s\partial t^{l-1}}  (\s_{\b\a}\o\bar{c}_i)(0,0)=(\s_{\b\a}\o\bar\g)^{(l)}(0)$ and  $\frac{\partial^l} {   \partial t^{l}}  (\s_{\b\a}\o\bar{c}_i)(0,0)=(\s_{\b\a}\o\bar\g)^{(l)}(0)$.

These last two equations yield
\begin{eqnarray*}
&&  d\s_{\b\a}(x)   \Big[ - \m_\a(x,\xi_1)\frac{\bar{\g}^{(i-1)}}{(i-2)!}      -
\dots    -        \stackrel{i-1}{M}_\a   \big(   x,\xi_1,\dots,\frac{\bar{\g}^{(i-1)}}{(i-1)!}   \big)   \bar{\g}'(0)  \Big]     \\
&=& +  d\s_{\b\a}(x)\frac{\bar\g^{(i)}(0)}{(i-1)!} -  \frac{(\s_{\b\a}\o\bar{\g})^{(i)}(0)}{(i-1)!}\\
&&-           \m_\b \Big(  \s_{\b\a}(x)  ,  (\s_{\b\a}\o\bar{\g})'(0)  \Big) \frac{      (\s_{\b\a}\o\bar{\g})^{(i-1)}(0)       }  {(i-2)!} - \dots\\
&&  - \stackrel{i-1}{M}_\b \Big(  \s_{\b\a}(x)  , \dots,   \frac{(\s_{\b\a}\o\bar{\g})^{(i-1)}(0)}  {(i-1)!}  \Big)  (\s_{\b\a}\o\bar{\g})' (0)
\end{eqnarray*}
\begin{eqnarray*}
&&     \hspace{-17 mm}         =- \frac{1}{(i-1)!}     \{     \sum_{j_1+j_2=i} a_{(j_1,j_2)}^i d^2\s_{\b\a}(x)[\bar\g^{(j_1)}(0) , \bar\g^{(j_2)}(0)  ]         +\dots  \\
&&     \hspace{-10 mm}     + d^{i}\s_{\b\a}(x)[  \bar{\g}'(0),\dots,\bar{\g}'(0)]                            \}\\
&&     \hspace{-10 mm}  -  \m_\b \Big(  \s_{\b\a}(x)  ,  (\s_{\b\a}\o\bar{\g})'(0)  \Big) \frac{      (\s_{\b\a}\o\bar{\g})^{(i-1)}(0)       }  {(i-2)!} - \dots\\
&&    \hspace{-10 mm}    - \stackrel{i-1}{M}_\b \Big(  \s_{\b\a}(x)  , \dots,   \frac{(\s_{\b\a}\o\bar{\g})^{(i-1)}(0)}  {(i-1)!}  \Big)  (\s_{\b\a}\o\bar{\g})' (0)
\end{eqnarray*}
\textbf{Step 3.}
As a consequence of steps 1 and 2 we get
$$
i z_\b^i([\g,x]_k)=id\s_{\b\a}(x)\xi_i \quad   ;  \quad2\leq i\leq k
$$
that is
\begin{eqnarray*}
\P_{\b\a}^k(x,\xi_1,\xi_2,\dots,\xi_k)=\Big(\s_{\b\a}(x),d\s_{\b\a}(x)\xi_1,\dots,d\s_{\b\a}(x)\xi_k\Big).
\end{eqnarray*}
This last means that  for any $\a,\b\in I$ with $U_{\b\a}\neq \emptyset$
\begin{eqnarray*}
\P_{\b\a}^k:U_{\b\a}  & \to & GL(\E^k)\\
x  &\mto&    \Big(d\s_{\b\a}(x)(.)  ,\dots,   d\s_{\b\a}(x)(.)\Big).
\end{eqnarray*}
is smooth. As a result, the family of trivializations  $\{  (\pi_k^{-1}(U_\a),\P_\a^k)  \}_{\a\in I}$ provides a vector bundle structure for $\pi_k:T^kM\to M$ with the fibres isomorphic to $\E^k$. Moreover, since $\pi_k:T^kM\to M$  and $\oplus_{i=1}^k \pi_1:  \oplus_{i=1}^k TM\to M$ have the same transition functions and fibres, then they are isomorphic vector bundles over $M$.
\end{proof}
%
%
%
%
%
%
%
%
The v.b. structure which is proposed in theorem \ref{osc k admits a vb} is affective in the following sense:
\begin{Pro}\label{lemma p k to p k-1}
Suppose that for some $k\geq 2$, $\pi_k:T^kM\to M$ admits a v.b. structure isomorphic to $\oplus_{i=1}^kTM$. Then
$\pi_{k-1}:T^{k-1}M\to M$ also possesses  a v.b. structure
isomorphic to $\oplus_{i=1}^{k-1}TM$.
\end{Pro}
\begin{proof}
Let $\{ (  {\pi_k}^{-1}(U_\a),  \P_\a^k)\}_{\a\in I}$ be a family of
trivializations for $\pi_k:T^kM\to M$ induced by the atlas
$\mathcal{A}=\{(U_\a,\s_\a)\}_{\a\in I}$ of $M$ as in theorem \ref{osc k admits a vb}. Then for any
$[\g,x]_k\in T^k_xM$ we have
\begin{equation*}
\P_\a^k([\g,x]_k)=\big(\s_\a(x),  (\s_\a\o\g)'(0)  ,  z_\a^2([\g,x]_k),\dots,z_\a^k([\g,x]_k)\big).
\end{equation*}
We claim that $\{({\pi_{k-1}}^{-1}(U_\a),  \P_\a^{k-1})\}_{\a\in
I}$ defines a v.b. structure on $\pi_{k-1}:T^{k-1}M\to M$ where
\begin{equation*}
\P_\a^{k-1}([\g,x]_k)=\big(\s_\a(x),  (\s_\a\o\g)'(0)  ,  z_\a^2([\g,x]_k),\dots,z_\a^{k-1}([\g,x]_{k-1})\big).
\end{equation*}
We show that  $\P_\a^{k-1}$ is bijective. In fact suppose that
$\P_\a^{k-1}([\g_1,x]_{k-1})=\P_\a^{k-1}([\g_2,x]_{k-1})$ then consider
representatives of the classes $[\g_1,x]_{k-1}$ and $[\g_2,x]_{k-1}$ such
that $\g_1^{(k)}(0)=\g_2^{(k)}(0)$. Since
$\P_\a^{k}([\g_1,x]_{k})=\P_\a^{k}([\g_2,x]_{k})$ then injectivity of
$\S_\a^k$ yields $[\g_1,x]_k=[\g_2,x]_k$ which means that
$[\g_1,x]_{k-1}=[\g_2,x]_{k-1}$.

Suppose that $(x,\xi_1,...,\xi_{k-1})\in \s_\a(U_\a)\times
\E^{k-1}$ be an arbitrary element. Since $\P_\a^k$ is bijective,
there exists $[\g,x]_k\in T^k_xM$ such taht
$$\P_\a^k([\g,x]_k)=(x,\xi_1,...,\xi_{k-1},0).$$ Clearly
$\P_\a^{k-1}([\g,x]_{k-1})=(x,\xi_1,...,\xi_{k-1})$ that is $\P_\a^{k-1}$ is surjective. Finally
\begin{eqnarray*}
\P_{\b\a}^{k-1}:\s_\a(U_{\b\a})&\to& GL(\E^{k-1})\\
x&\mto& \Big( \underbrace  {d\s_{\b\a}(x)(.),\dots,  d\s_{\b\a}(x)(.) }_{(k-1)-times} \Big)
\end{eqnarray*}
is smooth. According to proposition 1.2 page 45 of \cite{Lang},  we deduce that $\pi_{k-1}:T^{k-1}M\to M$ admits
a v.b. structure isomorphic to $\oplus_{i=1}^{k-1}TM$.
\end{proof}
%
%
If we restrict our attention to $C^k$-partitionable manifolds (see e.g. \cite{ali}) then, we have the following inverse for theorem \ref{osc k admits a vb}.
\begin{The}
Suppose that $k\geq 2$. If $\pi_k:T^kM\to M$ admits a v.b.
structure isomorphic to $\oplus_{i=1}^kTM$, then  a linear connection on $M$ can be defined.
\end{The}
\begin{proof}
For $k> 2$ one can iterate lemma \ref{lemma p k to p k-1} and
conclude that $\pi_2:T^2M\to M$ admits a v.b. structure
isomorphic to $TM\oplus TM$. Then according to  \cite{Dod-Gal} theorem 3.4
or \cite{ali} theorem 2.3 there exists a linear connection on $M$.
\end{proof}

\begin{Cor}
\textbf{i.} For $k\geq 2$, $\pi_k:T^kM\to M$ admits a v.b.
structure isomorphic to $\oplus_{i=1}^kTM$ if and only if $M$ is
endowed
with a linear connection.\\
\textbf{ii.} If for some $k\geq 2$, $\pi_k$ becomes a v.b.
isomorphic to $\oplus_{i=1}^kTM$ then for every $i\in \N$ the
tangent bundle $T^iM$ also admits a v.b. structure isomorphic to
$\oplus_{j=1}^iTM$.
\end{Cor}
\subsection{Lifting of a  Riemannian metric}
Invoking theorem \ref{osc k admits a vb} we  introduce an special
lift of a given Riemannian metric  $g$ from the base manifold $M$ to its higher order tangent bundle $T^kM$. Let
$\{(U_\a,\s_\a)\}_{\a\in I}$ be an atlas for $M$. Denote by
$g_\a:TU_\a\times TU_\a\to \R$ the local  representative of the
metric $g$ restricted to the chart $(U_\a,\s_\a)$. Fix $k\in \N$
and consider the v.b. trivializations introduced in theorem
\ref{osc k admits a vb}. For every $\a\in I$ define the  bilinear
symmetric form
\begin{equation*}
G_\a^k:\pi_k^{-1}(U_\a)\times \pi_k^{-1}(U_\a)\to \R
\end{equation*}
mapping $\Big([\g_1,x]_k,[\g_2,x]_k\Big)$ to
\begin{equation*}
\sum_{i=1}^kg_\a(x)\Big(proj_i\o\P_\a^k([\g_1,x]_k),proj_i\o\P_\a^k([\g_2,x]_k)\Big)
\end{equation*}
where $proj_i$ stands for the projection to the $(i+1)$'th factor.

Due to the transition functions $\P_{\b\a}^k$  the family
$\{G_\a\}_{\a\in I}$ defines a Riemannian metric on $T^kM$.
\begin{Rem}
In the case that $M$ is modeled on a Hilbert manifold (or a self
dual Banach space \cite{Lang}),  we deal with Riemannian metrics.
But if we go one step further then we will loose the
definiteness condition of our metrics.
\end{Rem}
\begin{Rem}\label{rem reducing the structure group to O(Ek)}
Let $M$ be a Hilbert manifold. As a result of theorem \ref{osc k admits a vb} and theorem 3.1, chapter VII \cite{Lang} we can assume  that a system of local trivializations $\{(\P_\a^k,\pi_k^{-1}(U_\a) )\}_{\a\in I}$ consists only orthogonal trivializations that is the transition maps take values in the orthogonal (or Hilbert)  group
$$\mathbb{O}(\E^k)=\{h\in GL(\E^k)~;~ \langle hv,hw \rangle=\langle v,w \rangle~~;~~v,w\in \E^k\}$$
(see also \cite{hfb}).
\end{Rem}

%
\subsection{Lifting of Lagrangians to Higher order tangent bundles}
In this section, using theorem \ref{osc k admits a vb}, we introduce a lift for Lagrangian form the base manifold to its higher order tangent bundles. To this end, we first  review the concepts of Lagrangian and Lagrangian vector field from \cite{Jer-Cher} and \cite{Miron}.

Let $M$ be smooth manifold modeled on the Banach space $\E$. A Lagrangian on $M$ is a smooth map $L:TM\to \R$ and the associated fibre derivative is the map
\begin{eqnarray*}
FL:TM   \to   T^*M
\end{eqnarray*}
where $FL(v)w=\frac{d}{dt}L(v+tw)|_{t=0}$ for any $v,w\in T_xM$.

\begin{Def}
A bilinear continuous map $B:\E\times \E\to\R$ is called weakly nondegenerate if for any $y\in\E$ the map $B^b:\E\to\E^*$; $B^b(y)z=B(y,z)$ is injective. We call $B$   nondegenerate (or strongly   nondegenerate) if $B^b$ is an isomorphism \cite{Jer-Cher}.
\end{Def}
Note that if $\E$ is a finite dimensional Banach space, then there is no difference between strong and weak nondegeneracy.

In a chart $(U_\a,\s_\a^1)$ of $TM$, let $L_\a$ represent $L$, that is $L_\a=L\o{\s_\a^1}^{-1}$. The Lagrangian $L$ is called (weakly) nondegenerate if for any chart $(U_\a,\s_\a^1)$, $\partial_2^2L_\a(x,y):L^2_{sym}(\E,\E)\to \R$ is (weakly) nondegenerate where $\partial_2$ denotes the partial derivative with respect to the second variable.

In finite dimensions this reads
\begin{equation*}
rank(g_{ij}(x,y))=rank \big(  \frac{1}{2}\frac{\partial^2L_\a(x,y)}{\partial y^i\partial y^j}  \big)_{1\leq i,j\leq dim (M)}=dim (M)
\end{equation*}
where $\s_\a^1=(x^i,y^i)_{1\leq i\leq n}$ is a local chart of $TM$.

We define the action of $L$  by $A:TM\to \R$, $A(v)=FL(v)v$ and the energy of $L$ is $E=A-L$. Locally we have
\begin{equation*}
E_\a(x,y)=\partial_2 L_\a(x,y)y-L_\a(x,y).
\end{equation*}

\begin{Def}
The vector field $Z_E\in\mathfrak{X}(TM)$ locally defined by
\begin{eqnarray*}
Z_E:TM|_{U_\a}  &\to &   TTM|_{U_\a}\\
(x,y)  &\mto&  (x,y,y,Z_\a(x,y))
\end{eqnarray*}
is called a Lagrangian vector filed for $L$ (\cite{Jer-Cher}) where
\begin{equation*}
Z_\a(x,y)={[\partial_2^2 L_\a(x,y)]}^{-1}  \big(   \partial_1L_\a(x,y)-\partial_1(\partial_2L_\p(x,y)y)  \big)
\end{equation*}
\end{Def}
It is easily seen that $Z_E$ is a second order vector field and the family $\{\m_\a=\partial_2 Z_a\}_{\a\in I}$ defines a connection on $M$. Then theorem \ref{osc k admits a vb} guarantees that $(\pi_k,T^kM,M)$ admits a vector bundle structure. A (weakly) nondegenerate lagrangian of order $k$ on $M$ is a differentiable map $L^k:T^kM\to M$ for which $\partial_{k+1}^2L_\a:L^2_{sym}(\E,\E)\to\R$, $\a\in I$, is (weakly) nondegenerate.

Let $L$ be a  nondegenerate Lagrangian on $M$. Then $L^k$ is a nondegenerate Lagrangian of order where
\begin{eqnarray*}\label{lifted lagrangian}
L_\a^k:\pi_k^{-1}(U_\a)   &   \to &  \R\\
\nonumber[\g,x]_k  &\mto &  \sum_{i=1}^kL_\a \big( \g_\a(0),proj_i\o\P_\a^k( [\g,x]_k ) \big) ~~;~~\a\in I
\end{eqnarray*}
(see also \cite{popescu} and \cite{Miron} for different lifted Lagrangians).
%
\section{Infinite order tangent bundle }\label{Section T infty M}
For any $x\in M$ and $\g_1,\g_2\in C_x$ define the \textbf{infinite
equivalence relation } denoted by $\approx^\infty$ as follows
\begin{equation*}
\g_1\approx^\infty_x \g_2 \textrm{ ~~if and only if for any }k\in \N,
~~\g_1\approx^k_x \g_2.
\end{equation*}
The equivalence class containing $\g$ is called an infinite tangent
vector at $x$ and is denoted by $[\g,x]_\infty$. Alternatively, we
may set $[\g,x]_\infty=\cap_{k=1}^\infty[\g,x]_k$ where the
intersection in non-empty since $\g$ belongs to $[\g,x]_k$ for any
$k\in \N$. The \textbf{infinite tangent space} at $x$, $T^\infty_x
M$ is defined to be $T^\infty_xM:=C_x/\approx^\infty_x$. The
infinite tangent bundle to $M$ is denoted by $T^\infty M$ where
$T^\infty M:=\cup_{x\in M}T^\infty_xM$. The canonical projection
$\pi_\infty:T^\infty M\to M$ projects the equivalence class
$[\g,x]_\infty$ onto $x$. If no confusion can rise, we write
$T^\infty M$ for both $T^\infty M$ as a manifold and $T^\infty M$
as a bundle over $M$.

We now propose a generalized Fr\'{e}chet manifold (v.b.) structure
for $T^\infty M$  that will be of aid in considering $T^\infty M$
as the projective limit of Banach  manifolds (v.b.'s) $T^kM$.

There are natural difficulties with Fr\'{e}chet manifolds, bundles
and even spaces. For example the pathological structure of the
general linear group on Fr\'{e}chet spaces puts in the question
defining a v.b. structure for $T^\infty M$ \cite{split}, \cite{GAl-VB}. Moreover there are serious drawbacks in the study of
differential equations on Fr\'{e}chet manifolds \cite{2de},
\cite{Hamilton}. In order to overcome these difficulties we will
use the projective limit tools to endow $T^\infty M$ with a
reasonable manifold and v.b. structure.

First we give some hints about a wide class of Fr\'{e}chet
manifolds i.e. those which may be considered as projective limits
of Banach manifolds (For more details see \cite{Dod-Gal} and the references therein). Let
$\{M^i,\phi^{ji}\}_{i,j\in\N}$ be a projective family of Banach
manifolds where the model spaces $\{\E^i\}_{i,\in\N}$, respectively,
also form a projective system of Banach spaces with the given
connecting morphisms $\{\r^{ji}:\E^j\to \E^i;~ j\geq
i\}_{i,j\in\N}$. Elements of $M:=\varprojlim M^i$ consist of all threads
${(x_i)}_{i\in \N}\in\prod_{i=1}^\infty M^i$ where
$\p^{ji}(x_j)=x_i$ for all $j\geq i$. Suppose that for every
thread ${(x_i)}_{i\in\N}\in\lim M^i$ there exists a projective
system of charts $\{U^i,\phi^{i}\}_{i\in\N}$ such that $x_i\in U^i$ and
$\lim U^i$ is open in $M:=\lim M^i$. Then $M$ admits a Fr\'{e}chet
manifold structure on $\F$ with the corresponding charts $\{\lim
U^i,\lim\p^i)\}$. Furthermore for any $i\in \N$ we have the
natural projections $\p_i:M\to M^i$; $(x_k)_{k\in\N}\mto x_i$ and
$\r_i:(e_k)_{k\in\N}\mto e_i$.

In our case  for any natural number $i$, $M^i:=T^iM$ and
$\E^i:=\overbrace{\E\times \E...\times \E}^{\textrm{i+1 times}}$
with the usual product norm $\parallel.\parallel_i$. For  $j\geq
i$, the connecting morphism $\p^{ji}:T^jM\to T^iM$ maps the class
$[\g,x]_j$ onto $[\g,x]_i$ and $\r^{ji}:\E^j\to \E^i$ is just
projection to the first $i+1$ factors. The canonical projective
systems of charts are
$\{\big({\pi_i}^{-1}(U_\a),\P_\a^i\big)\}_{i\in \N}$ rising from
theorem \ref{osc k admits a vb}. Consequently $T^\infty M$ admits
a \textbf{smooth Fr\'{e}chet manifold structure} modeled on the
Fr\'{e}chet space $\F=\lim\E^i\subseteq \prod_{k=1}^\infty$. Note that $\F$ is a Fr\'{e}chet
space with the associated metric
$$d(x,y)=\sum_{i=1}^\infty\frac{\parallel x_i-y_i
\parallel_i}{2^i(1+\parallel x_i-y_i\parallel_i)}$$
where
$x,y\in\F:=\varprojlim \E^i$, $x_i=\r_i(x)$, $y_i=\r_i(y)$ and
$\r_i:\F\to \E^i$; ${(x_k)}_{k\in\N}\mto x_i$ is  the canonical
projection.

In a further step we will try to supply $\pi_\infty:T^\infty M\to
M $ with a generalized v.b. structure. \\
Suppose that for any
$i\in\mathbb{N}$, $(\pi_i,E^i,M)$ be a Banach v.b. on $M$ with the
fibres of type $\mathbb{E}^i$ where
$\{\mathbb{E}^i,\r^{ji}\}_{i,j\in \mathbb{N}}$ also forms  a
projective system of Banach spaces. With these notations we state
the following definition from \cite{GAl-VB}.
\begin{Def}
The system $\{(\pi_i,E^i,M), f^{ji}\}_{i,j\in\mathbb{N}}$ is
called a \textbf{strong projective system of Banach v.b.'s}
over the same basis $M$ if;\\
i) $\{E^i, f^{ji}\}_{i,j\in\mathbb{N}}$ is a projective system of
Banach manifolds.\\
ii) For any ${(x^i)}_{i\in\mathbb{N}}\in F:=\lim E^i$, there
exists a projective system of trivializations
$\tau^i:{\pi_i}^{-1}(U)\longrightarrow U\times \mathbb{E}^i$  of
$(E^i,\pi_i,M)$ such that $x^i\in U\subseteq M$ and $(id_U\times\r^{ji})\circ\tau^j=\tau^i\circ
f^{ji}$ for all  $j\geq i$.
\end{Def}
Now the projective systems of v.b's is defined by setting
$E^i:=T^iM$, $\tau^i_\a:=(\s_\a^{-1}\times id_{\E^i})\o\P_\a^i$ and
$\p^{ji},\r^{ji}$ as in the previous part.

For any $i\in \N$
define the Banach Lie group
\begin{equation*}
\mathcal{H}^0_i(\mathbb{E}^i)=\{(l_1,\dots,l_i)\in\prod_{j=1}^i  \mathcal{L}(\mathbb{E}^j,\mathbb{E}^j);~~
\r^{jk}\circ l_j=l_k\circ\r^{jk}~\textrm{for all} ~k\leq j\leq
i\}\vspace{-2mm}
\end{equation*}
Using proposition 1.2 of \cite{GAl-VB} we conclude that
$\pi_\infty:T^\infty M\to M$ admits a generalized v.b. structure
over $M$ with fibres isomorphic to the Fr\'{e}chet space $\F=\lim
\E^i$ and the structure group ${\mathcal{H}}^0(\F):=\lim
\mathcal{H}^0_i(\mathbb{E}^i)$.
\begin{Rem}
In the case where $M$ is a finite dimensional manifold, $ T^\infty
M$ becomes a Fr\'{e}chet v.b. over $M$ with fibres isomorphic to
the known Fr\'{e}chet space $\R^\infty$.
\end{Rem}
%
%
\begin{Examp}
In this example we introduce  the restricted symplectic group $Sp_2(\mathcal{H})$  (\cite{Galvan}) and we propose a vector bundle structure for $\big( \pi_k,T^kSp_2(\mathcal{H})$, $Sp_2(\mathcal{H} \big)$ for $k\in\N\cup\{\infty\}$.
Let $(\mathcal{H}, \langle,\rangle)$ be an infinite dimensional real Hilbert space and $J$ be a complex structure on $\mathcal{H}$. The symplectic  group $Sp(\mathcal{H})$ is defined by
\begin{equation*}
Sp(\mathcal{H})=\{g\in GL(\mathcal{H}): g^*Jg=J\}.
\end{equation*}
The Lie algebra of $Sp(\mathcal{H})$ is
\begin{equation*}
\mathfrak{sp}(\mathcal{H})=\{x\in \mathcal{B}(\mathcal{H});~ xJ=-Jx^*\}
\end{equation*}
Denote by $\mathcal{B}_2(\mathcal{H})$ the Hilbert–-Schmidt class $B_2(\mathcal{H}) = \{g \in \mathcal{B}(\mathcal{H}) : Tr(g^*g) < \infty\}$
where $Tr$ is the usual trace and $\mathcal{B}(\mathcal{H})$ is the set of all bounded linear operators on $\mathcal{H}$.
Define the restricted symplectic group  to be
\begin{equation*}
Sp_2(\mathcal{H})=\{g\in Sp(\mathcal{H}): g-1\in B_2(\mathcal{H})\}
\end{equation*}
Then the Lie algebra of $ Sp_2(\mathcal{H})$ is
$\mathfrak{sp}_2(\mathcal{H})  =     \{   x\in B_2(\mathcal{H}): xJ=-Jx^*  \}$
which is a closed subspaces of $B_2(\mathcal{H})$ and hence a Hilbert space \cite{Galvan}. Moreover, for any $g\in Sp_2(\mathcal{H})$,
\begin{equation*}
(TSp_2(\mathcal{H}))_g=g\mathfrak{sp}_2(\mathcal{H})\subset B_2(\mathcal{H})
\end{equation*}
is an inner product space endowed with the left invariant Riemannian metric
\begin{equation}\label{left invariant metric}
\langle v,w \rangle_g= \langle g^{-1}v,g^{-1}w \rangle=Tr( (gg^*)^{-1}vw^*)  ~;~ ~ v,w\in T_gSp_2(\mathcal{H})
\end{equation}
However, the Riemannian connection on $Sp_2(\mathcal{H})$ is given by the local form (Christoffel symbol)
\begin{equation*}
2g^{-1}\Gamma_g(gx,gy)=xy+yx+x^*y+y^*x-xy^*-yx^*
\end{equation*}
for any $g\in Sp_2(\mathcal{H})$ and $x,y\in \mathfrak{sp}_2(\mathcal{H})$.

As a consequence  for any $k\in\N$, $\pi_k:T^kSp_2(\mathcal{H})\to Sp_2(\mathcal{H} )$ admits a vector bundle structure with fibres isomorphic to $\mathfrak{sp}_2(\mathcal{H})^k$ and the structure group $GL(\mathfrak{sp}_2(\mathcal{H}))$. Since the base manifold is a Riemannian manifolds, for $k\in\N$, the above vector bundle can be considered as a vector bundle with $\mathbb{O}(\mathfrak{sp}_2(\mathcal{H})^k)$ as its structure group (remark \ref{rem reducing the structure group to O(Ek)}).

Moreover $\pi_\infty:T^\infty Sp_2(\mathcal{H})\to Sp_2(\mathcal{H} )$ becomes a generalized vector bundle with fibres isomorphic to $\mathfrak{sp}_2(\mathcal{H})^\infty=\lim
\mathfrak{sp}_2(\mathcal{H})^i$ and the structure group ${\mathcal{H}}^0(\mathfrak{sp}_2(\mathcal{H})^\infty)$.
\end{Examp}
%
%
%


\end{document}